\documentclass[11pt,final]{article}
\usepackage{amsmath,amsthm,amsfonts,amssymb,graphicx,enumerate,psfrag}
\usepackage{mathrsfs}
\usepackage{fullpage}

\usepackage[colorlinks,citecolor=blue,urlcolor=blue]{hyperref}
\usepackage[utf8]{inputenc} 
\newtheorem{lemma}{Lemma}
\newtheorem{theorem}{Theorem}
\newtheorem{prop}{Proposition}
\newtheorem{example}{Example}
\newtheorem{corollary}{Corollary}
\newtheorem{remark}{Remark}

\newcommand{\dN}{\mathbb {N}}

\newcommand{\dZ}{\mathbb {Z}}
\newcommand{\dR}{\mathbb {R}}
\newcommand{\dG}{\mathbb {G}}
\newcommand{\cG}{\mathcal {G}}
\newcommand{\rr}{\mathcal {R}}
\newcommand{\rl}{\mathcal {L}}
\newcommand{\cS}{\mathcal {S}}
\newcommand{\cN}{\mathcal {N}}
\newcommand{\cF}{\mathcal {F}}
\newcommand{\cX}{\mathcal {X}}
\newcommand{\cU}{\mathcal {U}}
\newcommand{\trel}{{\rm t}_{\textsc{rel}}}
\newcommand{\tmix}{{\rm t}_{\textsc{mix}}}
\newcommand{\tsep}{{\rm t}_{\textsc{sep}}}
\newcommand{\EE}{{\mathbb{E}}}
\newcommand{\PP}{{\mathbb{P}}}
\newcommand{\bP}{{\mathbf{P}}}
\newcommand{\bE}{{\mathbf{E}}}

\newcommand{\s}{{\rm{s}}}
\newcommand{\Tfill}{T}
\newcommand{\sep}{{\textsc{sep}}}
\newcommand{\dd}{{\rm d}_{\sep}}
\newcommand{\dtv}{{\rm d}_{\textsc{tv}}}
\newcommand{\ARW}{{\rm ARW}}
\newcommand{\IDLA}{{\rm IDLA}}

\title{Separation cutoff for Activated Random Walks}
\author{Alexandre Bristiel (ENS Lyon) and Justin Salez (Paris-Dauphine)} 
\begin{document}

\maketitle
\begin{abstract}We consider Activated Random Walks on arbitrary finite networks, with particles being inserted at random and absorbed at the boundary. Despite the non-reversibility of the dynamics and the lack of knowledge on the stationary distribution, we explicitly determine the relaxation time of the process, and prove that separation cutoff is equivalent to the product condition. We also provide sharp estimates  on the center and width of the cutoff window. Finally, we illustrate those results by establishing explicit separation cutoffs on various networks, including: (i) large finite subgraphs of any fixed infinite non-amenable graph, with absorption at the boundary and (ii) large finite vertex-transitive graphs with absorption at a single vertex. The latter result settles a conjecture of Levine and Liang. Our proofs rely on the refined analysis of a strong stationary time recently discovered by Levine and Liang and involving the IDLA process.
\end{abstract}
\tableofcontents

\section{Introduction}

\subsection{Background and motivation}

Introduced a decade ago by Rolla and Sidoravicius \cite{ARWRoSi}, the  \emph{Activated Random Walks} (\ARW) process has quickly become the best candidate in the fascinating quest for a universal model of  \emph{self-organized criticality} \cite{MR1933220,SOC,Manna,Manna1991}.  
In a nutshell, this interacting particle system involves two species called \emph{active}
and \emph{sleeping}: active particles perform random walks and fall asleep at a constant rate $\lambda\in(0,\infty)$, while sleeping particles become active upon contact with another particle. On an infinite transitive graph $\dG$  such as $\mathbb \dZ^d$, a central question is that of \emph{stabilization}: if we start with one active particle at the origin and a Bernoulli($\zeta$) configuration of sleeping particles,  will the active species almost-surely die out? Rolla, Sidoravicius and
Zindy \cite{MR3956161} showed the existence of a critical density $ \zeta_c=\zeta_c(\dG)$ below which the answer is yes and above which it is no. This phase transition is moreover known to be non-trivial (i.e., $0<\zeta_c<1$) for many transitive graphs \cite{MR3813989}, including $\dZ^d$ for all $d\ge 1$ \cite{highdim,onedim,twodim}. We refer to the survey \cite{ARWsurvey} for additional details and references.

The present paper is concerned with the much less understood \emph{finitary setup}, recently considered by Basu, Ganguly, Hoffman and Richey \cite{MR4010935} or Levine and Liang \cite{ARWmixing}. Specifically, we restrict the \ARW\ dynamics to a large but finite domain $V$ in $\dG$ and kill  any particle that exits $V$. Clearly, if we add an active particle to a configuration of sleeping particles,  then the process will eventually stabilize to a new, random configuration of sleeping particles almost-surely. Iterating this procedure gives rise to a natural Markov chain on sleeping configurations, whose behavior is expected to be closely related to that of the unrestricted dynamics on $\dG$. In particular, it is believed to exhibit the following form of self-organized criticality: in a suitable double limit  where both the number of iterations and the size of the domain $V$ tend to infinity,  the proportion of occupied sites in $V$ should spontaneously concentrate around  the critical value $\zeta_c(\dG)$ for stabilization on the infinite graph $\dG$. Even more remarkably, the worst-case total-variation distance to equilibrium of the chain is expected to drop abruptly from $1-o(1)$ to $o(1)$ (a phase transition known as a  \emph{cutoff}) as the number of steps passes the critical threshold $|V|(\zeta_c(\dG)+o(1))$. See the recent work \cite{ARWmixing} and  the forthcoming paper \cite{Conj}  for details.

Motivated by those fascinating predictions, we here investigate the mixing properties of the  above Markov chain on arbitrary finite networks, when convergence to equilibrium is measured in the (stronger) \emph{separation distance}. In particular, we explicitly determine the relaxation time, we obtain sharp bounds on the separation mixing time, and we completely characterize the occurrence of a separation cutoff. Finally, we illustrate those results by establishing explicit  cutoffs on various concrete geometries. Before we can state our precise results, let us   define our setup more formally.

\subsection{Setup}
\paragraph{Geometry.} Throughout the paper, $V$ is a  finite set whose elements are called \emph{sites}. The geometry of our model is specified by a \emph{sub-stochastic matrix} on $V$, i.e., a function $K\colon V^2\to[0,1]$ such that
\begin{eqnarray*}
\forall x\in V, \quad \sum_{y\in V}K(x,y) & \le & 1.
\end{eqnarray*}
The latter naturally describes the evolution of a \emph{killed random walk}  (henceforth referred to as a $K-$walk) which, when alive at a site $x\in V$, jumps to a new site $y\in V$ at rate $K(x,y)$ and is killed at rate $1-\sum_{y\in V}K(x,y)$. We will always assume that the matrix $K$ is \emph{non-degenerate} in the sense that no principal sub-matrix of $K$ is stochastic: this guarantees that the $K-$walk is killed eventually almost-surely, regardless of where it starts. Here is a simple generic example to keep in mind, and to which we shall come back later.

\begin{example}[Restriction of an infinite graph]\label{ex:graph} Let  $\dG=(V_\dG,E_\dG)$ be a locally finite, infinite connected graph, and let $P_{\dG}\colon V^2_{\dG}\to[0,1]$ denote the associated random-walk  transition matrix, i.e.
\begin{eqnarray*}
 P_{\dG}(x,y) & := & \left\{
 \begin{array}{ll}
 \frac{1}{\deg(x)} & \textrm{ if }\{x,y\}\in E_{\dG}\\
 0 & \textrm{else}.
\end{array} 
 \right.
 \end{eqnarray*} 
 Now, choose a finite set $V\subseteq V_{\dG}$. Then, the restriction $K:=P_{\dG}[V]$ of $P_{\dG}$ to $V^2$  is a non-degenerate sub-stochastic matrix, which describes the evolution of a random walk on $\dG$ killed upon exiting $V$. 
\end{example}
\paragraph{Dynamics.} The instantaneous state of our interacting particle system is described by a \emph{configuration} $\eta\in\{0,\s,1,2,\ldots\}^V$, with the following interpretation: each site $x\in V$ is either empty ($\eta(x)=0$), or occupied by a single \emph{sleeping} particle ($\eta(x)=\s$), or occupied by an arbitrary number $k\ge 1$ of \emph{active} particles ($\eta(x)=k$). Given a parameter $\lambda\in(0,\infty)$, we then consider the following continuous-time Markovian evolution on the space of configurations:
\begin{itemize}
\item active particles perform independent $K-$walks, but fall asleep at rate $\lambda$ when alone on a site;
\item sleeping particles do not move, but become active upon contact with another particle.
 \end{itemize}
Note that any configuration $\eta$ without active particles (i.e., $\eta\in \{0,\s\}^V$) is stable under this dynamics. In fact, our non-degeneracy assumption on $K$ guarantees that  the process will eventually  be absorbed in such a  stable configuration almost-surely, from any   initial configuration. This random procedure is called \emph{stabilization}. Note that it may significantly reduce the total number of particles in the system, but can not augment it. Following \cite{ARWmixing}, we now  compensate it with a simple insertion mechanism to define an ergodic Markov chain on  stable configurations. 

\paragraph{The $\ARW(K,\lambda,\nu)$ chain.} Fix a probability distribution $\nu$ on $V$ whose support intersects every connected component of $K$, so that every site has a chance of being visited by a $K-$walk starting from $\nu$. A natural example to keep in mind is of course the uniform law on $V$, henceforth denoted by $\cU(V)$ or simply $\cU$. Finally,  consider the discrete-time Markov chain on the set  $\cX:=\{0,\s\}^V$ which, at each step,  turns the current state $\eta$ into a new, random state $\xi$ as follows: 
\begin{enumerate}
\item make $\eta$ unstable by adding one active particle at a $\nu-$distributed random site;
\item run the \ARW\ dynamics to obtain a (random) stable configuration $\xi$. 
\end{enumerate}
Write  $P\colon \cX^2\to [0,1]$ for the corresponding transition matrix. Our assumptions on the parameters $(K,\lambda,\nu)$ guarantee  that $P$ is irreducible and aperiodic. In particular, we have
\begin{eqnarray}
\label{mixing}
\forall \eta,\eta'\in\cX,\qquad 
P^t\left(\eta,\eta'\right) & \xrightarrow[t\to\infty]{} & \pi(\eta'),
\end{eqnarray}
where $\pi=\pi P$ denotes the unique invariant distribution of the chain. We note that the latter is far from explicit, and that very little is known about it, despite fascinating predictions \cite{ARWmixing,Conj}. We are here interested in the speed at which the convergence  (\ref{mixing}) occurs. Formalizing this requires a few  definitions, which we now briefly recall. We refer the unfamiliar reader to the book \cite{MixingBook} for  details.

\paragraph{Separation distance, relaxation time and mixing time.} There are many natural ways to measure the distance to equilibrium of a Markov chain at a given time $t\in\dN$. We will here focus on the so-called \emph{separation distance} (see Section \ref{sec:tv} for a comparison with the total-variation distance):
\begin{eqnarray*}
\dd(t) & := & \max_{\eta,\eta'\in\cX}\left(1-\frac{P^t(\eta,\eta')}{\pi(\eta')}\right).
\end{eqnarray*}
The function $\dd\colon\dN\to[0,1]$ is non-increasing, with $\dd(0)=1$ and $\dd(+\infty)=0$. In our case, it is moreover  sub-multiplicative (see Lemma \ref{lm:submult} below). By Gelfand's formula, we classically have
\begin{eqnarray}
\label{decayrate}
\left(\dd(t)\right)^{1/t} \ \xrightarrow[t\to\infty]{} \ \rho, \ & \textrm{ where } & 
\rho \ := \ \max\left\{|z|\colon z\in\mathbb C\setminus\{1\}\textrm{ is an eigenvalue of }P\right\}. 
\end{eqnarray} 
In light of this, a first natural time-scale for the Markov chain is the so-called  \emph{relaxation time}:
\begin{eqnarray*}
\trel & := & \frac{1}{1-\rho}.
\end{eqnarray*}
This ``near-equilibrium'' parameter  has, however, very little to say  about the early behavior of the chain, because the geometric decay $\dd(t)\approx \rho^{t}$ promised by (\ref{decayrate}) is only valid in the $t\to\infty$ limit.  In the ``far-from-equilibrium" regime, a more appropriate  time-scale  is  the (separation) \emph{mixing time}: 
\begin{eqnarray*}
\tsep(\varepsilon) & := & \inf\{t\in\dN\colon\dd(t)\le \varepsilon\}.
\end{eqnarray*}
The default value for the precision parameter  $\varepsilon\in(0,1)$  is $1/2$, in which case we simply write $\tsep=\tsep(1/2)$. Note  that replacing $1/2$ with any smaller value $\varepsilon$ can not increase $\tsep$ by more than the multiplicative factor $\log_2\left(\frac{1}{\varepsilon}\right)$,  thanks to the  sub-multiplicativity of $\dd$.
\subsection{Main results}
The high-level message of our paper is that the mixing behavior of the $\ARW(K,\lambda,\nu)$ chain is essentially dictated by a very simple single-particle statistics, namely, the hitting probabilities 
\begin{eqnarray}
\label{def:p}
p(x) & := & \bP_\nu\left(\tau_x<\tau_\dagger\right),\quad x\in V.
\end{eqnarray}
Here and throughout the paper, we use the notation $\bP_\nu$ for the law of a  $K-$walk  with initial distribution $\nu$ (henceforth referred to as a $(K,\nu)-$walk), and write $\tau_\dagger$ for its life-time and $\tau_x$ for the hitting time of  $x$. Intuitively, a small value of $p(x)$ means that the site $x$ is poorly connected. The presence of such ``remote sites'' constitutes a natural obstruction to mixing, because a particle sleeping there may have to wait for a very long time before being activated. Perhaps surprisingly, our results below will show that this is the main obstruction to mixing. 
\begin{remark}[Reversible network]\label{rk:rev}There is a useful alternative expression for $p(x)$ in the  case where
\begin{eqnarray}
\label{def:rev}
\forall x,y\in V,\quad \nu(y)K(y,x) & = & \nu(x)K(x,y).
 \end{eqnarray} 
Indeed, this identity remains true if we replace $K$ with the Green's function $\cG=I+K+K^2+\cdots$, and we may then divide through  by $\cG(x,x)$ and sum over all $y\in V$ to arrive at the identity
\begin{eqnarray*}
\forall x\in V,\quad p(x) & = & \frac{\nu(x)\bE_x[\tau_\dagger]}{\cG(x,x)}.
\end{eqnarray*}
Note that the detailed balance condition (\ref{def:rev}) always holds in Example \ref{ex:graph}, with $\nu(x)\propto\deg(x)$.
\end{remark} 

\paragraph{Relaxation time.}  Our first main result is the exact determination of the relaxation time of the $\ARW(K,\lambda,\nu)$ chain in full generality, and despite the total lack of reversibility of the dynamics. 
\begin{theorem}[Relaxation time]\label{th:trel}The relaxation time of the $\ARW(K,\lambda,\nu)$ chain is exactly given by
\begin{eqnarray*}
\trel & = & \left(\min_{x\in V}p(x)\right)^{-1}.
\end{eqnarray*}
\end{theorem}

\paragraph{Window and cutoff.} Our second main result concerns the \emph{cutoff phenomenon}, a remarkable phase transition in the convergence to equilibrium of certain processes \cite{MR1374011}. Specifically, a sequence of ergodic Markov chains (indexed by $n$) exhibits a  separation  cutoff if for any fixed $\alpha\ge 0$,
\begin{eqnarray*}
\dd^{(n)}\left(\lfloor\alpha \tsep^{(n)}\rfloor\right) & \xrightarrow[n\to\infty]{} & \left\{
\begin{array}{ll}
1& \textrm{ if }\alpha<1\\
0 & \textrm{ if }\alpha>1.
\end{array}
\right.
\end{eqnarray*}
An equivalent formulation is that the time-scale over which the distance to equilibrium drops from near $1$ to near $0$ (known as the \emph{mixing window}) is much smaller than the center of the window, i.e. 
\begin{eqnarray*}
\forall \varepsilon\in(0,1/2), \quad 
 \tsep^{(n)}(\varepsilon)-\tsep^{(n)}(1-\varepsilon) & = & o\left(\tsep^{(n)}\right).
\end{eqnarray*}
Here and throughout the paper, we use the standard notations $a_n=o(b_n)$,  $a_n=\Omega(b_n)$ and $a_n=O(b_n)$ to respectively mean that the sequence $({a_n}/{b_n})_{n\ge 1}$ vanishes as $n\to\infty$, is bounded away from $0$, and is bounded away from $+\infty$.
Establishing  a cutoff is a notoriously delicate task, which a priori requires the determination of $\tsep^{(n)}(\varepsilon)$ within a multiplicative precision $1+o(1)$. 
A well-known simple necessary condition  is the so-called \emph{product condition}:
\begin{eqnarray}
\label{prod}
\trel^{(n)} & = & o\left(\tsep^{(n)}\right).
\end{eqnarray}
Unfortunately, the latter is too weak to guarantee   cutoff in general (see \cite[Example 18.7]{MixingBook} for a generic counter-example). Nevertheless, the condition (\ref{prod}) has been shown to guarantee separation cutoff for  birth-and-death chains \cite{MR2288715}, and for exclusion processes with reservoirs \cite{exclusion}. Our second main result adds all \ARW\ chains to this short list, and provides an estimate on the mixing window.
\begin{theorem}[Mixing window and cutoff]\label{th:cutoff}For any $\varepsilon\in(0,1/2)$, the $\ARW(K,\lambda,\nu)$ chain satisfies
 \begin{eqnarray*}
 \tsep(\varepsilon)-\tsep(1-\varepsilon) & \le & \sqrt{\frac{8\trel \tsep}\varepsilon}.
 \end{eqnarray*}
In particular, a sequence of $\ARW$ chains exhibits cutoff if and only if it satisfies (\ref{prod}).
\end{theorem}

\paragraph{Mixing time.}Our third main result is a  two-sided estimate on the mixing time  of the $\ARW(K,\lambda,\nu)$ chain,  which will be shown to be sharp in many concrete examples. The estimate involves two simple statistics $\rl$ and $\rr$, which are related to  the hitting probabilities $(p(x)\colon x\in V)$ as follows:
\begin{eqnarray*}
\rl \ := \ \left(\frac 1{|V|}{\sum_{x\in V}p(x)}\right)^{-1}, & \qquad & 
\rr \ := \ |V|\max_{x\in V}\left\{\frac{\bP_{\cU}(\tau_x<\tau_\dagger)}{p(x)}\right\}.
\end{eqnarray*}
Observe that the sum $\sum_{x\in V}p(x)$ is exactly the expected \emph{range} (number of distinct visited vertices) of a $(K,\nu)-$walk, and is thus at most the expected life-time $\bE_\nu[\tau_\dagger]$. In other words,
\begin{eqnarray}
\label{range}
\rl & \ge & \frac{|V|}{\bE_\nu[\tau_\dagger]}.
\end{eqnarray}
Note also that the parameter $\rr$ compares the actual hitting probabilities $(p(x)\colon x\in V)$ with those that would be obtained if $\nu$ was the uniform law $\cU$. In particular, we have $\rr=|V|$ when $\nu=\cU$.

\begin{theorem}[Sharp mixing-time estimates]\label{th:tsep}We have $\rl \ \le  \trel  \le \ \rr$ and
\begin{eqnarray*}
\max\left\{\trel,|V|,\frac{\rl \log |V|}{5}\right\} \ \le & \tsep & \le \ \left(\sqrt{\rr}+3\sqrt{\trel\log |V|}\right)^2.
\end{eqnarray*}
\end{theorem}
Those lower and upper bounds match in the following two generic situations: for each $n\ge 1$, consider an $\ARW(K_n,\lambda_n,\nu_n)$ chain and write $\trel^{(n)},\tsep^{(n)},\rr^{(n)},\rl^{(n)}$ for the associated statistics. 
\begin{itemize}
\item If  $\rl^{(n)},\rr^{(n)}=\Theta(|V_n|)$, then $\trel^{(n)}=\Theta(|V_n|)$ and  $\tsep^{(n)}=\Theta\left(|V_n|\log |V_n|\right)$.
\item If $\nu_n=\cU(V_n)$ and $\trel^{(n)}=o\left(\frac{|V_n|}{\log |V_n|}\right)$, then $\tsep^{(n)}=|V_n|+o(|V_n|)$.
\end{itemize}
Let us now illustrate those results  by providing explicit  cutoffs on several  concrete geometries.  
\subsection{Examples}

\paragraph{Non-amenable graphs.} Our first application concerns the  setup of Example \ref{ex:graph}. Let $\dG$ be an infinite  graph with bounded degrees, and recall that $\dG$ is called \emph{non-amenable} (see, e.g., \cite{Woess}) if 
\begin{eqnarray*}
\label{def:nonamenable}
\inf\left\{\frac{|\partial V|}{|V|}\colon V\subseteq V_{\dG},\quad 0< |V|<\infty\right\} & > & 0.
\end{eqnarray*}
where  $\partial V$ is the set of vertices in $V_\dG\setminus V$ having a neighbor in $V$. A simple example is the infinite $3-$regular tree. By virtue of Cheeger inequalities, non-amenability translates into the existence of a spectral gap for  the random-walk operator $P_\dG$: there exists  $\gamma_\dG\in(0,1)$ such that
\begin{eqnarray*}
\sum_{x,y\in V_\dG} \deg(x)P_\dG^t(x,y)f(x)f(y) & \le & (1-\gamma_\dG)^t \sum_{x\in V_\dG} \deg(x)f^2(x),
\end{eqnarray*}
for any $t\in\dN$ and any function $f\colon V_{\dG}\to\dR$ such that the right-hand side is finite. Taking $f={\bf 1}_V$ for some finite set $V\subseteq V_{\dG}$, we deduce that the random walk started from  the degree-biased  law  $\nu(x)\propto \deg(x){\bf 1}_V(x)$ and killed upon exiting $V$ satisfies  
\begin{eqnarray*}
\bE_\nu\left[\tau_\dagger\right] & \le &  \frac{1}{\gamma_\dG}.
\end{eqnarray*}
 We may then replace $\nu$ with the uniform law $\cU$ by paying a factor equal to the maximum degree of $\dG$. In view of (\ref{range}) and Theorems \ref{th:cutoff} and \ref{th:tsep}, we deduce the following general result.
\begin{corollary}[Cutoff on non-amenable graphs]\label{co:exp}
Let $\dG=(\mathbb V,\mathbb E)$ be an infinite non-amenable graph with bounded degrees. For each $n\ge 1$, choose an arbitrary $n-$element  subset $V_n\subseteq\mathbb V$, and set $K_n=P_{\dG}[V_n]$, $\nu_n=\cU(V_n)$, and $\lambda_n\in(0,\infty)$. Then, the $\ARW(K_n,\lambda_n,\nu_n)$ chain satisfies
\begin{eqnarray*}
\trel^{(n)} \ = \ \Theta(n), & \qquad & \tsep^{(n)} \ =\ \Theta(n\log n),
\end{eqnarray*}
and there is a separation cutoff with window $O(n\sqrt{\log n})$.
\end{corollary}
\paragraph{Vertex-transitive graphs with a sink.} Our second application  confirms a conjecture of Levine and Liang \cite[Conjecture 9]{ARWmixing}. Let $G$ be a finite connected graph. Recall that  $G$ is \emph{vertex-transitive} if for any vertices $x,y\in V_G$, there is an edge-preserving bijection $\phi\colon V_G\to V_G$ that maps $x$ to $y$. In words, $G$ ``looks the same from every vertex''.   Under this assumption, various random-walk statistics admit considerably simplified expressions, as recorded by Aldous \cite{AldousHitting}. In particular, Proposition 3 therein states that the random walk on $G$ started from the uniform law satisfies
\begin{eqnarray*}
\PP\left(\tau_x<\tau_z\right)& = & \frac{1}{2},
\end{eqnarray*}
for any $x\ne z\in V_G$. 
Note that  the left-hand side is exactly $\frac{|V_G|-1}{|V_G|}p(x)$, where $p(x)$ is the hitting probability of $x$ by a random walk starting from $\nu=\cU(V_G\setminus\{z\})$ and killed upon hitting $z$. 
\begin{corollary}[Cutoff on transitive graphs with a sink]\label{co:trans}
For each $n\in\dN$, let $G_n$ be a connected vertex-transitive graph on $n+1$ vertices, and let $V_n$ be obtained by removing one vertex. Take $K_n=P_{G_n}[V_n]$, $\nu_n=\cU(V_n)$ and  $\lambda_n\in(0,\infty)$. Then, the $\ARW(K_n,\lambda_n,\nu_n)$ chain satisfies 
\begin{eqnarray*}
\trel^{(n)} \ = \ \frac{2n}{n+1}, & \qquad & \tsep^{(n)}\ = \ n+o(n),
\end{eqnarray*}
and there is a separation cutoff with window $O(\sqrt{n})$.
\end{corollary}
\paragraph{Wheel-like graphs.} Our third application is motivated by the following example of Levine and Liang: start from the $n-$cycle on $V_n=[n]$, and connect all sites to an extra vertex called the \emph{sink}. Write $K_n$ for the transition matrix of  random walk killed upon hitting the sink, and set $\nu=\cU(V_n)$ and $\lambda_n\in(0,\infty)$. Then \cite[Proposition 8]{ARWmixing} states that the  $\ARW(K_n,\nu_n,\lambda_n)$ chain satisfies 
\begin{eqnarray*}
\tsep^{(n)} & = & \Omega\left(\frac{n\log n}{\log\log n}\right).
\end{eqnarray*}
Now, observe that the life-time of a $(K,\nu)-$walk is here a geometric variable with mean $3$,  hence $\rl^{(n)}\ge n/3$. Moreover,  $\rr^{(n)}=n$ because $\nu_n$ is uniform. Thus, our general results imply that in fact,  $\tsep^{(n)}=\Theta(n\log n)$ and  $\trel^{(n)}=\Theta(n)$, and that there is a separation cutoff with window $O(n\sqrt{\log n})$. Moreover, the same argument  applies to any bounded-degree graph instead of the cycle. 
\begin{corollary}[Cutoff on wheel-like graphs]\label{co:wheel} Fix  $d\in\dN$. For each $n\ge 1$, let $G_n$ be a graph on $n+1$ vertices with one vertex of degree $n$ (the ``sink'') and all others of degree at most $d$. Let $V_n$ be the set obtained by removing the sink, and set $K_n=P_{G_n}[V_n]$, $\nu_n=\cU(V_n)$ and $\lambda_n\in(0,\infty)$. Then, 
\begin{eqnarray*}
\trel^{(n)} \ = \ \Theta(n), & \qquad & \tsep^{(n)} \ = \ \Theta(n\log n),
\end{eqnarray*}
and there is a separation cutoff with window $O(n\sqrt{\log n})$. 
\end{corollary}

\paragraph{Discrete Euclidean balls.} Finally, let us revisit the important case of discrete Euclidean balls analyzed by Levine and Liang in \cite{ARWmixing}. Consider the setup of Example \ref{ex:graph} where the ambient graph $\mathbb G$  is the $d-$dimensional Euclidean lattice $\mathbb Z^d$, and where the finite domain $V$ is the  ball
\begin{eqnarray*}
V_n & := & \left\{x\in\dZ^d\colon x_1^2+\cdots+x_d^2 \le n\right\}.
\end{eqnarray*}
Choose $\nu_n:=\cU(V_n)$ and note that the reversibility condition in Remark \ref{rk:rev} holds. Consequently, classical Green's function estimates show  that $\min_{x\in V_n}p(x)\ge \frac{c_d}{g_d(n)}$, where $c_d>0$ is a constant and
\begin{eqnarray*}
g_d(n) & := & \left\{
\begin{array}{ll}
1 & \textrm{if }d=1;\\
n \log n & \textrm{if }d= 2;\\
n^{d-1} & \textrm{if }d\ge 3.
\end{array}
\right.
\end{eqnarray*}
Thus, the  $\ARW(V_n,\lambda_n,\nu_n)$  chain satisfies $\trel^{(n)}=O\left(g_d(n)\right)$. Keeping in mind that  $|V_n|=\Theta(n^{d})$, we deduce that $\trel^{(n)}=o\left(\frac{|V_n|}{\log(V_n)}\right)$ in all dimensions, hence the following result.
\begin{corollary}[Cutoff on Euclidean balls]The  $\ARW(V_n,\lambda_n,\nu_n)$  chain described above satisfies
\begin{eqnarray*}
\trel^{(n)} \ = \ O\left(g_d(n)\right), & \quad & \tsep^{(n)}=|V_n|+o(|V_n|), 
\end{eqnarray*}
and there is a separation cutoff with window $O(\sqrt{n^{d}g_d(n)}).$
\end{corollary}
\subsection{Implications for total-variation distance}
\label{sec:tv}
 Let us finally discuss what our results imply for the more standard \emph{total-variation mixing time}:
\begin{eqnarray*}
\tmix(\varepsilon) \ := \ \min\left\{t\in\dN\colon \dtv(t)\le \varepsilon\right\},& \textrm{ where } &  \dtv(t)\ :=\ \max_{\eta\in \cX,A\subseteq\cX}\left|P^t(\eta,A)-\pi(A)\right|.
\end{eqnarray*}
We first note that the asymptotic behavior (\ref{decayrate}) remains valid in total variation, so that Theorem \ref{th:trel} also characterizes the asymptotic behavior of the function $t\mapsto \dtv(t)$. Moreover, we always have $\dtv(t)\le \dd(t)$  (see, e.g., \cite[Lemma 6.16]{MixingBook}), so the upper bound in Theorem \ref{th:tsep} also applies to $\tmix$.  Although the converse relation fails for non-reversible chains,  the lower bound $\tmix\ge \trel$ remains valid, as it  follows from the sub-multiplicativity of $t\mapsto 2\dd(t)$. Let us sum this up.
\begin{corollary}[Total-variation mixing time]For \ARW chains we always have
\begin{eqnarray*}
\trel\ \le & \tmix & \le \ \left(\sqrt{\rr}+3\sqrt{\trel\log |V|}\right)^2.
\end{eqnarray*} 
\end{corollary}
Regarding total-variation cutoff, we recall that the latter implies  $\trel^{(n)}=o(\tmix^{(n)})$, which in turns implies $\trel^{(n)}=o(\tsep^{(n)})$ and hence separation cutoff, by our Theorem \ref{th:cutoff}. This result is non-trivial, in the sense that it does not hold for general chains \cite{MR3530321}. We thus record it here. 
\begin{corollary}[Total-variation cutoff implies separation cutoff]If a sequence of $\ARW$ chains exhibits total-variation cutoff, then it also exhibits separation cutoff. 
\end{corollary}
Finally, we  note that that the total-variation and separation distances become equivalent if the sleeping rate is large enough. Indeed, it is not hard to see that
\begin{eqnarray*}
\sup_{t\in\dN}\left|1-\frac{\dtv(t)}{\dd(t)}\right| & \le & \frac{|V|}{1+\lambda}.
\end{eqnarray*} 
\begin{corollary}[Deep-sleep regime]
In the  regime  $\lambda_n\gg |V_n|$, the various  examples of separation cutoffs presented above extend to total-variation cutoffs, with the same locations and windows.
\end{corollary}
Of course, establishing total-variation cutoff for a fixed sleeping rate $\lambda\in(0,\infty)$ in any of the examples presented above remains a challenging open problem, and we once again refer the interested reader to the fascinating predictions formulated in \cite{ARWmixing,Conj}.

 \section{Proofs}
 
The remainder of the paper is organized as follows. In Section \ref{sec:reduc}, we build on the  work of Levine and Liang \cite{ARWmixing} to reduce the study of Activated Random Walks to that of  Internal Diffusion Limited Aggregation. In Section \ref{sec:coupling}, we construct a Markovian grand coupling of the latter with two crucial properties:  monotonicity, and  a kind of  concavity. Finally we use those properties to establish Theorems \ref{th:trel}, \ref{th:cutoff} and \ref{th:tsep} in Section \ref{sec:trel}, \ref{sec:cutoff} and \ref{sec:tsep} respectively. 
 \subsection{Reduction to IDLA}
\label{sec:reduc}
The starting point of our analysis is a beautiful connection, recently uncovered by Levine and Liang \cite{ARWmixing}, between the mixing properties of the $\ARW(K,\lambda,\nu)$ chain and the filling time of a classical growth model known as  \emph{Internal Diffusion Limited Aggregation} (\IDLA) \cite{MR1218674}. The latter may be thought of as the $\lambda\to \infty$ limit of the $\ARW$ chain, where active particles instantaneously fall asleep when alone on a site. More formally, the $\IDLA(K,\nu)$ process is the discrete-time Markov chain  $(\cS_t)_{t\ge 0}$ on subsets of $V$  constructed as follows: let $W_1,W_2,\ldots$ be a  sequence of i.i.d. $(K,\nu)-$walks, and let  $\cS_0\subseteq\cS_1\subseteq\cS_2 \ldots$ be defined inductively by $\cS_0:=\emptyset$ and 
$\cS_{t}:=f\left(\cS_{t-1},W_t\right)$, where 
\begin{eqnarray}
\label{def:IDLA}
f\left(S,(x_0,\ldots,x_r)\right) & := &
\left\{ 
\begin{array}{ll}
S \cup\{x_k\} & \textrm{if }k=\min\{i\le r\colon x_i\notin S\}\textrm{ exists;}\\
S & \textrm{otherwise.}
\end{array}
\right.
\end{eqnarray}
In words, particles enter the system one after the other and, upon arrival, each performs a $(K,\nu)-$walk until it dies or hits a previously unoccupied site, where it settles forever. The main quantity of interest for us will be the so-called \emph{filling time}
\begin{eqnarray*}
\Tfill & := & \min\{t\in\dN\colon \cS_t=V\},
\end{eqnarray*}
which is almost-surely finite by our running assumptions on $(K,\nu)$. It turns out that the tail of $\Tfill$ controls the distance to equilibrium of the  \ARW\ chain. More precisely, Levine and Liang \cite{ARWmixing} observed that under a natural coupling due to Shellef  \cite{Coupling}, the filling time of $\IDLA(K,\nu)$ is actually a \emph{strong stationary time} for the $\ARW(K,\lambda,\nu)$ chain, implying (see \cite[Chapter 6]{MixingBook})  the bound 
 \begin{eqnarray*}
\forall t\in\dN,\qquad \dd(t) & \le & \PP\left(\Tfill>t\right).
\end{eqnarray*}
Our only contribution in this section is the  observation that the strong stationary time $T$ is in fact \emph{optimal}, in the sense that the above inequality is  an equality. 
\begin{prop}[\ARW\ vs \IDLA]\label{pr:reduction}For any choice of the parameters $(K,\lambda,\nu)$, we have
\begin{eqnarray*}
\forall t\in\dN,\qquad \dd(t) & = & \PP\left(\Tfill>t\right).
\end{eqnarray*}
\end{prop}
\begin{proof}We only briefly revisit the coupling argument used by Levine and Lieng, and refer to \cite{ARWmixing} for more details. We will denote by $\bf 0$ (resp. $\bf 1$, resp. $\rm\bf s$) the configuration in which all entries are equal to $0$ (resp. $1$, resp. $\s$). An elementary but crucial property of \ARW\ stabilization is that it can be performed sequentially, by selecting an arbitrary active particle at each step and letting it execute a random transition (moving, dying, or falling sleep) with the appropriate distribution, until no active particle remains. The key point is that the rule used for selecting active particles is irrelevant: this is the so-called \emph{Abelian property} (see \cite{ARWRoSi,Abelian,ARWmixing} for details).   In particular, in order to stabilize a configuration $\eta\in\{0,1,\s,2,3,\ldots\}^V$, we may proceed in two stages as follows:
\begin{enumerate}[(i)]
\item Only select particles which are not alone on their sites,  until some $\zeta\in\{0,\s,1\}^V$ is reached. 
\item Stabilize $\zeta$ to produce the desired stable configuration $\xi\in \{0,\s\}^V$. 
\end{enumerate}
Note that Stage (i) preserves the coordinate-wise order on configurations induced by the natural single-site ordering $0\preceq \s\preceq 1\preceq 2\preceq\ldots$. More precisely, under an obvious coupling, we have
\begin{eqnarray}
\label{MONO}
 \eta\preceq\eta' & \Longrightarrow & \zeta\preceq\zeta'.
\end{eqnarray}
In particular, on the event $\{\zeta={\bf 1}\}$, we must also have $\zeta'={\bf 1}$ and hence $\xi'=\xi$. This insensitivity to the initial condition easily implies (see the proof of \cite[Theorem 1]{ARWmixing} for details) that 
\begin{eqnarray}
\label{SST}
\emph{\textrm{the conditional distribution of  $\xi$ given $\{\zeta=\bf 1\}$ is exactly the stationary law $\pi$}}.
\end{eqnarray}
Now, fix $t\in\dN$ and an initial stable configuration $\eta\in\{0,\s\}^V$. Thanks to the Abelian property, we may generate a random configuration $\xi_t$ with law $P^t(\eta,\cdot)$  as follows: we first make $\eta$ unstable by independently inserting $t$ active particles according to $\nu$, and we then run the two-stage stabilization process described above to  produce $\zeta_t\in\{0,\s,1\}^V$  and  $\xi_t\in\{0,\s\}^V$. Writing $\PP_\eta$ to explicitate the dependency in the initial condition $\eta$, we then have for any $\eta'\in \{0,\s\}^V$, 
\begin{eqnarray*}
P^t(\eta,\eta') & = &  \PP_\eta\left(\xi_t=\eta'\right)\\
& \ge & \PP_\eta\left(\zeta_t={\bf 1}, \xi_t=\eta'\right)\\
& = & \PP_\eta\left(\zeta_t={\bf 1}\right)\pi(\eta')\\
& \ge & \PP_{\bf 0}\left(\zeta_t={\bf 1}\right)\pi(\eta')
\end{eqnarray*}
where we have used (\ref{SST})  and then (\ref{MONO}). Moreover, the first  inequality is  an equality in the special case $(\eta,\eta')=(\bf 0,{\rm \bf s})$, because no particle can fall asleep during Stage (i) and the number of particles can not increase during Stage (ii). Also, the second inequality is trivially an equality when $\eta=\bf 0$. Recalling the definition of $\dd(t)$, we have thus shown that
\begin{eqnarray*}
\dd(t) & = &  1-\frac{{P^t(\bf 0},{\rm \bf s})}{\pi({\rm \bf s})} \ = \ \PP_{\bf 0}\left(\zeta_t\ne {\bf 1}\right).
\end{eqnarray*}
This is the desired identity, because under the initial condition $\eta=\bf 0$, the random configuration $\zeta_t$ is distributed exactly as ${\bf 1}_{\cS_t}$. To see this, simply recall the $\lambda\to\infty$ interpretation of $\IDLA(K,\nu)$, and observe  that sending $\lambda\to\infty$ in the two-stage stabilization process does not affect $\zeta_t$ (no particle can fall asleep during Stage (i)) but reduces  $\xi_t$ to the unique stable configuration with the same support as $\zeta_t$ (no particle has time to move or die during Stage (ii)). 
\end{proof}

 \subsection{Monotonicity and concavity}
 \label{sec:coupling}
In view of Proposition \ref{pr:reduction}, we may now completely forget our $\ARW(K,\nu,\lambda)$ chain, and focus on a new problem of independent interest: understanding the distribution of the filling time  of the $\IDLA(K,\nu)$ process. To this end, it will be convenient to consider multiple \IDLA\ chains starting from all possible initial conditions and coupled by using the same sequence $W_1,W_2,\ldots $ of i.i.d. $(K,\nu)$-walks. More precisely,  for each $A\subseteq V$, we define $(\cS^A_t)_{t\ge 0}$ inductively by 
\begin{eqnarray*}
\cS_0^A & := & A \\
\cS_{t}^A & := & f\left(\cS_{t-1}^A,W_t\right)\textrm{ for }t\ge 1,
\end{eqnarray*}
where  $f$ is the update function defined at (\ref{def:IDLA}). We then consider the associated filling time:
\begin{eqnarray*}
T^A & := & \inf\{t\ge 0\colon \cS_t^A=V\}.
\end{eqnarray*}
In particular, $T^V=0$ and for any $x\in V$, the random variable $T^{V\setminus \{x\}}$ is just the number of independent $(K,\nu)$-walks that one needs to sample in order to hit $x$, i.e.
\begin{eqnarray}
\label{geom}
T^{V\setminus \{x\}} & \sim & \textrm{Geometric}(p(x)).
\end{eqnarray}
To be consistent with our previous notation, we will simply write $\cS_t^\emptyset=\cS_t$ and $T^\emptyset=T$. 
The interest of this Markovian grand coupling is revealed in the following lemma. While the monotonicity property (i) is a well-known feature of \IDLA, the subtler concavity property (ii) seems to be new, and will play a crucial role in our proof of Theorem \ref{th:cutoff}. 
\begin{lemma}[Monotonicity and concavity]\label{lm:mono} For every $t\in\dN$ and every $A\subseteq B\subseteq V$, we have
\begin{enumerate}[(i)]
\item $\cS_t^A  \subseteq  \cS_t^B$
\item $|\cS_{t+1}^B|-|\cS_{t+1}^A|  \le  |\cS_{t}^B|-|\cS_{t}^A|$.
  \end{enumerate}
\end{lemma}
\begin{proof} Fix $A\subseteq B\subseteq V$ and a finite sequence of sites $w$,  and let us compare $A':=f(A,w)$ with $B':=f(B,w)$. In view of the definition of $f$, there are four possible cases:
\begin{enumerate}
\item If $w$ does not visit $A^c$,  then $(A',B')=(A,B)$.
\item If $w$ visits $A^c$ but not $B^c$, then 
$(A',B')=(A\cup\{x\},B)$ for some $x\in B\setminus A$.
\item If $w$ visits  $B^c$ before  $B\setminus A$, then $(A',B')=(A\cup\{x\},B\cup\{x\})$ for some $x\in B^c$.
\item If $w$ visits  $B^c$ after $B\setminus A$, then $(A',B')=(A\cup\{x\},B\cup\{y\})$ for some $x\in B\setminus A$ and  $y\in B^c$.
\end{enumerate}
In all cases, we have $A'\subseteq B'$ and $|B'|-|A'|\le |B|-|A|$, so (iii) and (iv) follow by induction.
\end{proof}
As a first application, let us establish the promised sub-multiplicativity of $\dd$.
\begin{lemma}[Sub-multiplicativity of $\dd$]\label{lm:submult}For all $s,t\in\dN$, we have 
\begin{eqnarray*}
\dd(s+t) &  \le  & \dd(s)\dd(t).
\end{eqnarray*} 
\end{lemma}
\begin{proof}Fix $s,t\in\dN$. Property (i) guarantees that the function $g_t\colon 2^V\to[0,\infty)$ defined by 
\begin{eqnarray*}
g_t(A) & := & \PP\left(T^A>t\right)
\end{eqnarray*}
is non-increasing. Now, by the Markov property at time $s$, we have
 \begin{eqnarray*}
\PP(\Tfill>s+t) & = & \EE\left[g_t(\cS_s){\bf 1}_{(\Tfill>s)}\right]\\
& \le & \EE\left[g_t(\emptyset){\bf 1}_{(\Tfill>s)}\right]\\
& = & \PP(\Tfill>t)\PP(\Tfill>s).
 \end{eqnarray*}
In view of Proposition \ref{pr:reduction}, this concludes the proof.
\end{proof}
A useful  consequence of this is that the mean filling time $\EE[T]$ is a $2-$approximation of $\tsep$.
\begin{lemma}[First-moment estimate]\label{lm:L1} We have
\begin{eqnarray*}
\frac{\EE[T]}{2} \ \le & \tsep & \le \ 2\EE[T].
\end{eqnarray*}
\end{lemma}
\begin{proof}Using the definition of $\tsep=\tsep(1/2)$ and Proposition \ref{pr:reduction}, we have
\begin{eqnarray*}
\frac{1}{2} & < & \dd(\tsep-1) \ = \ \PP(T\ge \tsep) \ \le \ \frac{\EE[T]}{\tsep},
\end{eqnarray*}
which yields the upper bound. The lower bound uses the  sub-multiplicativity of $t\mapsto\PP(T>t)$ to relate the mean and median of $T$: we have $\PP(T>q\tsep)\le 2^{-q}$ for all $q\in\dN$,  hence
\begin{eqnarray*}
\EE[T] & =  & \sum_{k=0}^\infty \PP(T>k)\\
& = & \sum_{q=0}^\infty\sum_{r=0}^{\tsep-1}\PP(T>q\tsep+r)\\
& \le & \sum_{q=0}^\infty\sum_{r=0}^{\tsep-1}2^{-q}\\
& = & 2\tsep.
\end{eqnarray*}
\end{proof}

\subsection{Proof of Theorem \ref{th:trel}}
\label{sec:trel}

Let us prove Theorem \ref{th:trel}. Write $p_\star:=\min_xp(x)$. In view of (\ref{decayrate}), our goal is to prove that 
\begin{eqnarray}
\label{trel:goal}
\lim_{t\to\infty}\left(\PP(\Tfill> t)\right)^{\frac 1t} & = & 1-p_\star.
\end{eqnarray}
Note that the limit on the left-hand side exists, by sub-multiplicativity. 
Now, fix $x\in V$ and $t\in\dN$. By the monotonicity of IDLA (Property (i) in Lemma \ref{lm:mono}) and (\ref{geom}), we have
\begin{eqnarray}
\label{trel:lb}
\PP(\Tfill> t) & \ge & \PP(\Tfill^{V\setminus\{x\}}> t)\ = \ (1-p(x))^t.
\end{eqnarray}
Optimizing over $x\in V$ already yields the lower bound  in (\ref{trel:goal}). Conversely, let us decompose $T $ as 
\begin{eqnarray*}
\Tfill & = & \sum_{k=0}^{n-1}(T_{k+1}-T_{k}),
\end{eqnarray*}
 where 
$
T_k   :=  \inf\left\{t\ge 0\colon |\cS_t|=k\right\}
$
denotes the first time at which $k$ distinct vertices have been covered. Note that the latter is a stopping time for  the natural filtration $(\cF_t)_{t\ge 0}$ of our Markovian grand coupling. Now fix $0\le k<n$ and observe that conditionally on $\cF_{T_k}$, the random variable $T_{k+1}-T_k$ has a geometric distribution  with success probability $p(\cS_{T_k}^c)$, where we have extended the definition of hitting probabilities (\ref{def:p}) from sites to subsets  $A\subseteq V$ in the natural way, i.e.,
\begin{eqnarray}
\label{def:pA}
p(A) & := & \bP_\nu\left(\tau_A<\tau_\dagger\right).
\end{eqnarray}
Since the function $p\colon 2^V\to[0,1]$ is non-decreasing, we have $p(A)\ge p_\star$ for any non-empty set $A$, hence for $A=\cS_{T_k}^c$. It follows that $\Tfill$ is stochastically dominated by the sum of $n$ independent geometric random variables with success parameter $p_\star$. In particular, 
\begin{eqnarray*}
\EE\left[z^{-\Tfill}\right] & < & \infty,
\end{eqnarray*}
for any $z\in(1-p_\star,1)$, and since $\PP(\Tfill>t)
 \le  z^{t} \EE\left[z^{-\Tfill}\right],$ the upper bound  in (\ref{trel:goal}) follows.

 \subsection{Proof of Theorem \ref{th:cutoff}}
 \label{sec:cutoff}
In this section, we estimate the width of the mixing window. We start by observing that the representation $\dd(t)=\PP(T>t)$ provided by Proposition \ref{pr:reduction} reduces our task to a  variance estimate.
\begin{lemma}[Width of the mixing window]\label{lm:window} For any $\varepsilon\in(0,1)$, we have
\begin{eqnarray*}
\tsep(\varepsilon)-\tsep(1-\varepsilon) & \le & 2\sqrt{\frac{\mathrm{Var}(T)}{\varepsilon}}.
\end{eqnarray*}
\end{lemma}
\begin{proof}Fix $\varepsilon\in(0,1)$ and let us introduce the integers
\begin{eqnarray*}
t^- \ := \ \left\lceil\EE[T]-\sqrt{\frac{\mathrm{Var}(T)}{\varepsilon}}\right\rceil\ & \textrm{ and } & t^+ \ := \ \left\lfloor\EE[T]+\sqrt{\frac{\mathrm{Var}(T)}{\varepsilon}}\right\rfloor.
\end{eqnarray*}
Using Chebychev's inequality, and the fact that $T,t^\pm$ are integers, we have
\begin{eqnarray*}
\PP(T>t^+) \ < \ \varepsilon & \textrm{ and } & \PP(T\ge t^-)\ > \ 1-\varepsilon.
\end{eqnarray*}
In view of Proposition \ref{pr:reduction}, this gives $\tsep(\varepsilon)\le t^+$ and $\tsep(1-\varepsilon)\ge t^-$, and we conclude that $\tsep(\varepsilon)-\tsep(1-\varepsilon)\le t^+-t^-$.
\end{proof}
We now recall a beautiful variance estimate for hitting times of increasing Markov processes due to Aldous \cite[Lemma 1.1]{AldousConcentration}. Surprisingly, the latter does not seem to have found many applications. 
\begin{lemma}[Variance estimate for hitting times, Aldous \cite{AldousConcentration}]\label{lm:aldous}Let $T$ be the hitting time of some fixed state by a Markov chain  on some finite state space. Write $h(A):=\EE_A[T]$ for its expectation when starting from the initial state $A$, and suppose that for any allowed transition $A\to B$, we have
\begin{eqnarray*}
0 \ \le & h(A)-h(B) & \le \ \kappa,
\end{eqnarray*}
where $\kappa$ is a constant. Then, $\textrm{Var}_A(T)\le \kappa\,\EE_A[T]$ for any initial state  $A$.
\end{lemma}
We will of course apply this lemma to the \IDLA\ process, and we thus  set
\begin{eqnarray*}
h(A) & := & \EE\left[T^A\right].
\end{eqnarray*}
Our next step consists in showing that the assumption of the above lemma holds with $\kappa=\trel$. 
\begin{lemma}[Control on increments]\label{lm:increment}For any $x\in V$ and any $A\subseteq V$, we have
\begin{eqnarray*}
0 & \le  \ h(A)-h(A\cup\{x\}) \ \le & \trel.
\end{eqnarray*}
\end{lemma}
\begin{proof}Fix $A\subseteq V$ and $x\in V$. The monotonicity of \IDLA\ (Property (i) in  Lemma \ref{lm:mono}) ensures that
\begin{eqnarray*}
\Tfill^{A\cup\{x\}} & \le  & \Tfill^A,
\end{eqnarray*}
which readily yields the lower bound. We now turn to the upper bound. By the strong Markov property at the stopping time $\tau=\Tfill^{A\cup\{x\}}$, we have
 \begin{eqnarray}
 \label{aldous}
 \EE\left[\Tfill^A\right] & =& \EE\left[\Tfill^{A\cup\{x\}}\right]+  \EE\left[h\left(\cS^A_{\tau}\right)\right].
 \end{eqnarray}
Now,  the concavity of \IDLA\ (Property (ii) in Lemma \ref{lm:mono}) implies
\begin{eqnarray*}
\forall t\in\dN,\quad \big|\cS^{A\cup\{x\}}_t\big|\le \left|\cS^A_t\right| +1,
\end{eqnarray*}
and choosing $t=\Tfill^{A\cup\{x\}}$ shows that the random set $\cS_\tau^A$ appearing in (\ref{aldous})  satisfies 
\begin{eqnarray*}
|\cS_\tau^A| & \ge & |V|-1.
\end{eqnarray*}
To conclude the proof, we  note that this last condition deterministically implies  $h(\cS_\tau^A)\le \trel$. Indeed,  we have $h(V)=0$ and  $h(V\setminus\{x\})=\frac{1}{p(x)}$ for every $x\in V$, by (\ref{geom}).
\end{proof}
We now have all we need to establish Theorem \ref{th:cutoff}.
\begin{proof}[Proof of Theorem \ref{th:cutoff}]
Combining Lemmas  \ref{lm:aldous} and \ref{lm:increment}, we have 
\begin{eqnarray*}
\mathrm{Var}(T) & \le & \EE[T]\trel\\
& \le & 2\trel\tsep,
\end{eqnarray*}
where the second line uses Lemma \ref{lm:L1}. In view of Lemma \ref{lm:window}, this concludes the proof.
\end{proof}

 \subsection{Proof of Theorem \ref{th:tsep}}\label{sec:tsep}
In this final section, we estimate $\EE[T]$ from above and below to prove Theorem \ref{th:tsep}. We write $n=|V|$. 
\begin{lemma}[A general lower bound]\label{lm:lower}We have
\begin{eqnarray*}
\EE[\Tfill] & \ge & \frac{n(\log n-\log\log n-1)}{\sum_{x\in V}p(x)}.
\end{eqnarray*}
\end{lemma}
\begin{proof}Fix an initial condition $A\subsetneq V$. As in the proof of Theorem \ref{th:trel}, we may decompose $T^A$ as
\begin{eqnarray*}
\Tfill^A & = & \sum_{k=|A|}^{n-1}(T_{k+1}^A-T_{k}^A),
\end{eqnarray*}
 where 
$
T_k^A   :=  \inf\left\{t\ge 0\colon |\cS_t^A|=k\right\}.
$
Now, fix $k\in\{|A|,\ldots,n-1\}$ and observe that conditionally on $\cF_{T_k^A}$, the random variable $T_{k+1}^A-T_k^A$ has a geometric distribution  with success probability $p\left(V\setminus\cS_{T_k^A}^A\right)$, where $p\colon 2^V\to[0,1]$ was defined at (\ref{def:pA}). In particular, we have
\begin{eqnarray*}
\EE\left[\Tfill^A\right] & = & \sum_{k=|A|}^{n-1} \EE\left[\frac{1}{p\left(V\setminus\cS_{{T_k^A}}^A\right)}\right].
\end{eqnarray*}
Now, it readily follows from its definition that the function $p$ is sub-additive: 
\begin{eqnarray*}
\forall B\subseteq V,\qquad p(B) & \le & \sum_{x\in B}p(x).
\end{eqnarray*}
Note also that the right-hand side can be further bounded by $|B|\max_{x\in B}p(x)$. 
Applying this to the random subset $B=V\setminus\cS_{{T_k^A}}^A$ appearing in the previous identity, we obtain
\begin{eqnarray*}
\EE\left[\Tfill^A\right] & \ge & \frac{1}{\max_{x\in A^c}p(x)}\sum_{k=|A|}^{n-1} \frac{1}{n-k}\\
& \ge & \frac{\log(1+|A^c|)}{\max_{x\in A^c}p(x)}.
\end{eqnarray*}
Finally, recall that $\EE[\Tfill]\ge \EE[\Tfill^A]$ by the monotonicity of \IDLA\ (Property (i) in Lemma \ref{lm:mono}). Since $A$ was arbitrary, we may finally replace it with its complement to get the nicer-looking bound:
\begin{eqnarray*}
\EE[\Tfill]
& \ge & \frac{\log(1+|A|)}{\max_{x\in A}p(x)}.
\end{eqnarray*}
It remains to optimize over the choice of the subset $A$. To do so, let us relabel the sites so that $V=[n]=\{1,\ldots,n\}$ and $p(1)\le \ldots\le p(n)$. Choosing $A=[k]$ for some $1\le k\le n$, we then have
\begin{eqnarray*}
\max_{x\in A}p(x) & = & p(k) \ \le \ \frac{p(k)+\cdots+p(n)}{n+1-k}\ \le \ \frac{p(1)+\cdots+p(n)}{n+1-k}.
\end{eqnarray*}
Thus, we arrive at
\begin{eqnarray*}
\EE[\Tfill]
& \ge & \frac{(n+1-k)\log(1+k)}{p(1)+\cdots+p(n)},
\end{eqnarray*}
and we may finally choose $k=\lceil \frac{n}{\log n}\rceil$ to optimize this bound and conclude the proof.
\end{proof}

\begin{lemma}[A general upper bound] \label{lm:upper}We have
\begin{eqnarray*}
 \tsep & \le & \left(\sqrt{\rr}+3\sqrt{{\trel\log n}}\right)^2.
\end{eqnarray*}
\end{lemma}
\begin{proof}We refine an argument first used in \cite{IDLA} in the case where $\nu$ is a Dirac mass, and then adapted in \cite{ARWmixing} to the case where $\nu$ is the uniform distribution. Fix a site $x\in V$, and let
\begin{eqnarray*}
T_x  & := &  \inf\{t\ge 0\colon x\in\cS_t\},
\end{eqnarray*}
be the first time that $x$ is occupied by $\IDLA$. By construction, the walk $W_{T_x}$ must hit $x$, and we let $Z_x$ denote its trajectory after the first visit to $x$. Note that this part of the trajectory is completely ignored in our construction  of $(\cS_t)_{t\ge 0}$. Moreover, by the Markov property, $Z_x$ is just a $K-$walk starting at $x$, and  the  walks $(Z_x\colon x\in V)$ are  independent. Now, fix $z\in V$ and let
\begin{eqnarray*}
\cN_\star(z) & := & \sum_{x\in V\setminus\{z\}}{\bf 1}_{\left\{\tau_z(Z_x)<\tau_\dagger(Z_x)\right\}},
\end{eqnarray*}
count the number of {ignored walks} that visit $z$. For $t\in\dN$, let also 
\begin{eqnarray*}
\cN_{t}(z) & := & \sum_{k=1}^t{\bf 1}_{\left\{\tau_z(W_k)<\tau_\dagger(W_k)\right\}},
\end{eqnarray*}
count those walks $W_1,\ldots,W_t$ that  hit $z$. Note that any visit to $z$ by a walk $W_k$  will result in  $\cS_k=\cS_{k-1}\cup\{z\}$, unless it occurs in the ignored part of $W_k$. Consequently, 
\begin{eqnarray*}
\PP\left(z\notin\cS_t\right) & \le & \PP\left(\cN_t(z)\le \cN_\star(z)\right).
\end{eqnarray*}
But $\cN_\star(z)$ and $\cN_t(z)$ are sums of independent Bernoulli variables, with 
\begin{eqnarray*}
\EE[\cN_t(z)] & = & tp(z)\\
 \EE[\cN_\star(z)] & = & \sum_{x\in V\setminus\{z\}}\bP_x(\tau_z<\tau_\dagger)\  \le \ \rr p(z).
\end{eqnarray*}
For $t\ge \rr$, we  have
$\sqrt{\EE[\cN_t(z)]}-\sqrt{\EE[\cN_\star(z)]}  \ge  \sqrt{p(z)}\left(\sqrt{t}-\sqrt{\rr}\right)$,
and Lemma \ref{lm:win} below gives
\begin{eqnarray*}
\PP(\cN_t(z)\le \cN_\star(z)) & \le & 2e^{-\frac{p(z)}{1+\sqrt{2}}\left(\sqrt{t}-\sqrt{\rr}\right)^2}.
\end{eqnarray*}
Taking a union bound over $z\in V$, we conclude that
\begin{eqnarray*}
\PP(T>t) & \le & 2\sum_{z\in V}e^{-\frac{p(z)}{1+\sqrt{2}}\left(\sqrt{t}-\sqrt{\rr}\right)^2}\\
& \le & 2ne^{-\frac{\left(\sqrt{t}-\sqrt{\rr}\right)^2}{(1+\sqrt{2})\trel}},
\end{eqnarray*}
where the second line uses Theorem \ref{th:trel}. This shows that
\begin{eqnarray*}
\tsep & \le & 1+\left(\sqrt{{\rr}}+\sqrt{2(1+\sqrt{2})\trel\log n}\right)^2\\
& \le & \left(\sqrt{{\rr}}+3\sqrt{\trel\log n}\right)^2,
\end{eqnarray*}
as desired.
\end{proof}

\begin{lemma}[Probability of an unexpected win]\label{lm:win}Let $U$ and $V$ be jointly defined random variables, each distributed as a sum of independent Bernoulli variables, and such that $\EE[U]\le\EE[V]$. Then,
\begin{eqnarray*}
\PP\left(V\le U\right) & \le & 2\exp\left\{-\frac{\left(\sqrt{\EE[V]}- \sqrt{\EE[U]}\right)^2}{1+\sqrt{2}}\right\}.
\end{eqnarray*}
\end{lemma}
\begin{proof}Set $u=\EE[U]$,  $v=\EE[V]$ and  $w:=\left(\frac{\sqrt{v}-\sqrt{u}}{2+\sqrt{2}}\right)^2$. Suppose we are given $a,b\ge 0$  satisfying
\begin{eqnarray}
\label{goal:win}
u+a  & \le & v-b.
\end{eqnarray}
This guarantees the inclusion $\{V\le U\}\subseteq \{U\ge u+b\}\cup \{V\le v-a\}$ and therefore
\begin{eqnarray*}
\PP\left(U\ge V\right) & \le & \PP\left(U\ge u+a\right)+\PP\left(V\le v-b\right) \\
& \le &   e^{-\frac{a^2}{2u+a}}+e^{-\frac{b^2}{2v}},
\end{eqnarray*}
where the second line uses Chernov bounds. To obtain the desired conclusion, we would like to set each of the two terms on the second line to $e^{-2w}$. In other words, we take
\begin{eqnarray*}
a &:= & w+\sqrt{w^2+4uw}\\
b &:= & 2\sqrt{vw}.
\end{eqnarray*}
It only remains to prove that this choice satisfies (\ref{goal:win}). Using $\sqrt{x+y}\le\sqrt{x}+\sqrt{y}$, we can write
\begin{eqnarray*}
u-v+a+b & = & u-v+2\sqrt{vw}+w+\sqrt{w^2+4uw}\\
& \le & u-v+2\sqrt{vw}+2w+2\sqrt{uw}\\
& = &\left(\sqrt{u}+\sqrt{v}\right)^2+2w-\left(\sqrt{v}-\sqrt{w}\right)^2\\
& \le & \left(\sqrt{u}+\sqrt{w}+\sqrt{2w}\right)^2-\left(\sqrt{v}-\sqrt{w}\right)^2\\
& = & 0,
\end{eqnarray*}
thanks to our careful choice of $w$.
\end{proof}

\begin{proof}[Proof of Theorem \ref{th:tsep}]The inequalities $\rl\le \trel\le\rr$ readily follow from  Theorem \ref{th:trel} and the  definitions of $\rl,\rr$. The upper bound on $\tsep$ is exactly Lemma \ref{lm:upper}. The lower bound $\tsep\ge n$ follows from the  observation that $\PP(\Tfill\ge n)=1$, because \IDLA\ grows by at most one site at each step. For the bound $\tsep\ge \trel$, we take $t=\tsep$ in the following argument, valid for all $t\in \dN$:
\begin{eqnarray*}
1-\frac{t}{\trel} & \le & \left(1-\frac{1}{\trel}\right)^{t}  \ = \ \left(1-p_\star\right)^t \ \le\  \dd(t),
\end{eqnarray*}
where the last inequality is (\ref{trel:lb}). Finally, the bound $\tsep\ge \frac{\rl \log n}{5}$ follows from  Lemmas \ref{lm:L1} and \ref{lm:lower} when $n\ge 100$, because we then have
\begin{eqnarray*}
\log n -\log\log n-1  & \ge & \frac{2\log n}{5}.
\end{eqnarray*}
For the case $n<100$, we instead simply use $\tsep\ge n\ge \frac{n\log n}{5} \ge \frac{\rl \log n}{5}$.
\end{proof}

\bibliographystyle{plain}
\bibliography{draft}
\end{document}